\documentclass[11pt]{amsart}

\usepackage{paralist, amsmath, amsthm, amssymb, color, graphicx}
\usepackage[margin=1.4in]{geometry}
\usepackage{hyperref}

\newtheorem{thm}{Theorem}

\newtheorem{prop}[thm]{Proposition}
\newtheorem{lemma}[thm]{Lemma}
\newtheorem{cor}[thm]{Corollary} 
\newtheorem{corollary}[thm]{Corollary}

\newtheorem{conj}[thm]{Conjecture}

\newcommand{\al}{\alpha}
\newcommand{\be}{\beta}
\newcommand{\ga}{\gamma}

\newcommand{\eps}{\epsilon}
\newcommand{\si}{\sigma}
\newcommand{\la}{\lambda}

\newcommand{\ov}{\overline}

\newcommand{\dir}{\textrm{dir}}
\newcommand{\spin}{\textrm{spin}}
\newcommand{\Id}{\textrm{Id}}

\newcommand{\ds}{\displaystyle}

\newcommand{\fig}[3]{\begin{figure}[h]\begin{center}\includegraphics[#1]{#2}\end{center}\caption{#3}\label{fig:#2}\end{figure}}

\newcommand{\ZZ}{\mathbf{Z}}

\newcommand{\mC}{\mathcal{C}}
\newcommand{\mF}{\mathcal{F}}

\newcommand{\mT}{\mathcal{T}}

\newcommand{\xx}{\textbf{x}}
\newcommand{\yy}{\textbf{y}}
\newcommand{\qq}{\textbf{q}}
\newcommand{\mm}{\textbf{m}}
\newcommand{\nn}{\textbf{n}}

\newcommand{\tw}{\widetilde{w}}
\newcommand{\hw}{\widehat{w}}
\newcommand{\ow}{\overline{w}}

\title{On the spanning trees of the hypercube and other products of graphs}

\author{Olivier Bernardi}
\thanks{I acknowledge support from NSF grant DMS-1068626, ERC ExploreMaps and ANR A3.}
\date{\today}

\begin{document}
\setcounter{tocdepth}{2}

\begin{abstract}
We give two combinatorial proofs of an elegant product formula for the number of spanning trees of the $n$-dimensional hypercube.
The first proof is based on the assertion that if one chooses a uniformly random rooted spanning tree of the hypercube and orient each edge from parent to child, then the parallel edges of the hypercube get orientations which are independent of one another. This independence property actually holds in a more general context and has intriguing consequences.
The second proof uses some ``killing involutions'' in order to identify the factors in the product formula. It leads to an enumerative formula for the spanning trees of the $n$-dimensional hypercube augmented with diagonals edges, counted according to the number of edges of each type. 
We also discuss more general formulas, obtained using a matrix-tree approach, for the number of spanning trees of the Cartesian product of complete graphs.
\end{abstract}


\maketitle

\section{Introduction}
Let $C_n$ be the hypercube in dimension $n$. The vertex set of $C_n$ is $\{0,1\}^n$, and two vertices are adjacent if they differ on one coordinate. It is known that the number of spanning trees of $C_n$ is 
\begin{equation}\label{eq:spanning-cube}
T(C_n)=\frac{1}{2^n}\prod_{i=1}^n(2i)^{n \choose i}.
\end{equation}
This formula can be obtained by using the matrix-tree theorem and then determining the eigenvalues of the Laplacian of the hypercube. This, in turns, can be done either using the representation theory of Abelian groups (applied to the group $(\ZZ/2\ZZ)^n$)~\cite[chapter 5]{Stanley:vol2}, or by guessing and checking a set of eigenvectors~\cite{Martin-Reiner:spanning-cube}. We will give two combinatorial proofs of this result, thereby answering an open problem mentioned for instance in~\cite[pp. 62]{Stanley:vol2} and~\cite{Hurlbert:spanning-trees} (the case $n=3$ was actually solved in \cite{Tuffley:spanning-3-cube} by a method different from ours).

We shall discuss refinements and generalizations of~\eqref{eq:spanning-cube} which are best stated in terms of rooted (spanning) forests. A \emph{rooted forest} of a graph $G$ is a subgraph containing every vertex such that each connected component is a tree with a vertex marked as the \emph{root vertex} of that tree. Given a rooted forest, we consider its edges as oriented in such a way that every tree is directed toward its root vertex (equivalently, every non-root vertex has one outgoing edge in $F$); see Figure~\ref{fig:cube-with-diagonals}.
We say that an oriented edge $e=(u,v)$ of the hypercube $C_n$ has \emph{direction} $i\in[n]$ and \emph{spin} $\eps\in\{0,1\}$ if the vertex $v$ is obtained from $u$ by changing the $i$th coordinate from $1-\eps$ to $\eps$. We denote $\dir(e)$ and $\spin(e)$ the direction and spin of the oriented edge $e$.  We then define a generating function in the variables $t$ and $\xx=(x_{1,0},x_{1,1},\ldots,x_{n,0},x_{n,1})$ for the rooted forests of the hypercube $C_n$ as follows:
\begin{equation}\label{def:enumeratorCn}
F_{C_n}(t;\xx):=\sum_{F\textrm{ rooted forest of } C_n}t^{\#\textrm{trees in } F}\prod_{e\in F} x_{\dir(e),\spin(e)}.
\end{equation}
We shall give two combinatorial proofs of the following result:
\begin{equation}\label{eq:spanning-cube-2}
F_{C_n}(t;\xx)=\prod_{S\subseteq [n]}\left(t+\sum_{i\in S} x_{i,0}+x_{i,1}\right).
\end{equation}
Note that the generating function of rooted spanning trees of $C_n$ is simply obtained by extracting the terms which are linear in $t$ in $F_{C_n}(t;\xx)$, so that~\eqref{eq:spanning-cube-2} gives a refinement of~\eqref{eq:spanning-cube}. This refinement was first proved by Martin and Reiner in~\cite{Martin-Reiner:spanning-cube} (in a slightly different form; see Section~\ref{sec:conclusion}) using a matrix-tree method. Using a similar method we shall give a generalization of~\eqref{eq:spanning-cube-2} valid for Cartesian products of complete graphs. However, the main goal of the present article is rather to investigate the combinatorial properties of the rooted forests of the hypercube suggested by~\eqref{eq:spanning-cube-2}.

\fig{width=.7\linewidth}{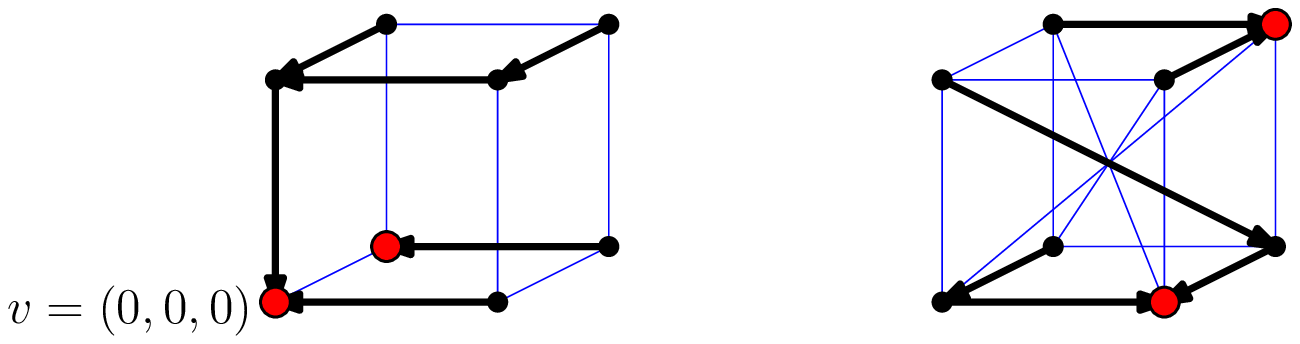}{A rooted forest of the cube $C_3$ (left) and of the cube augmented with its main diagonals $D_3$ (right). The root vertices are indicated by large dots.}

In Section~\ref{sec:bunkbed}, we prove a surprising \emph{independence} property for the spins of the parallel edges of a random rooted forest of $C_n$. More precisely, we show that for a uniformly random rooted forest conditioned to have exactly $k$ trees and to contain a given set of edges in direction $n$, the spins of these edges are \emph{independent} and uniform. This property is illustrated in Figure~\ref{fig:example-bunkbed-lemma}.
The independence of the spins remains true when conditioning the forest to have a given number $n_{i,\eps}$ of edges with direction $i$ and spin $\eps$ for all $i\in[n-1]$ and $\eps\in\{0,1\}$, and holds in the more general context of so-called \emph{bunkbed graphs}. 
Using the independence property, it is not hard to prove~\eqref{eq:spanning-cube-2}.

\fig{width=\linewidth}{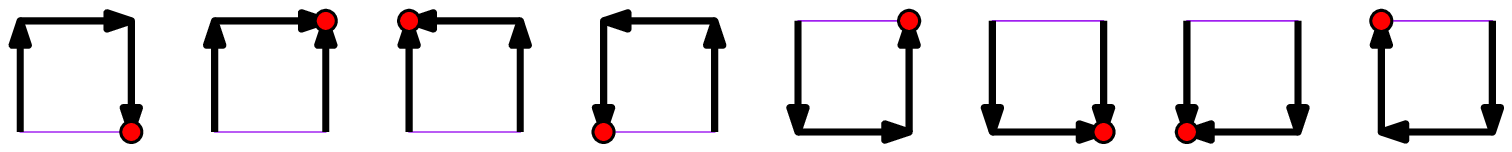}{The rooted spanning trees of the square containing both vertical edges. The spins of the vertical edges are seen to be independent.}

In Section~\ref{sec:withdiago}, we consider the graph $D_n$ obtained by adding the \emph{main diagonals} to the hypercube $C_n$, that is, the edges joining each vertex to the antipodal vertex; see Figure~\ref{fig:cube-with-diagonals}. We prove a generalization of~\eqref{eq:spanning-cube-2} for the generating function $F_{D_n}(t;\xx,y)$ of rooted forests of $D_n$, where the variable $y$ counts the number of diagonal edges contained in the forests (so that $F_{C_n}(t;\xx)=F_{D_n}(t;\xx,0)$). Our strategy there is to determine combinatorially the roots of $F_{D_n}(t;\xx)$ considered as a polynomial in $t$ by exhibiting some ``killing involutions'' for the rooted forests of $D_n$.

In Section~\ref{sec:matrix-tree}, we establish a generalization of~\eqref{eq:spanning-cube-2} valid for Cartesian products of complete graphs using the matrix-tree theorem. Finally, we gather some additional remarks and open questions in Section~\ref{sec:conclusion}.\\


We end this introduction with a few definitions. 
We call \emph{digraph} a finite directed graph. 
We denote a digraph $G=(V,A)$ to indicate that $V$ is the set of vertices and $A$ is the set of arcs, and for an arc $a\in A$ we denote $a=(u,v)$ to indicate that the arc $a$ goes from the vertex $u$ to the vertex $v$. We shall identify undirected simple graphs with the digraphs obtained by replacing each edge by two arcs in opposite directions. 

A \emph{rooted forest} of a digraph $G$ is a subgraph without cycle, containing every vertex and such that each vertex is incident to at most one outgoing arc. We call \emph{root vertices} the vertices not incident to any outgoing arc. Hence, each connected component of a rooted forest is a tree directed toward its unique root vertex. Two rooted forests are represented in Figure~\ref{fig:cube-with-diagonals}. We denote by $k(F)$ the number of connected components of the forest $F$. A digraph is \emph{weighted} if every arc $a$ has a weight $w(a)$ (which can be an arbitrary variable). The \emph{weight} of a forest $F$ is $w(F)=\prod_{a\in F}w(a)$ where the product is over the arcs contained in $F$. 
The \emph{forest enumerator} of a weighted digraph $G$ is 
\begin{equation}\label{eq:forest-enumerator}
F_G(t)=\sum_{F\textrm{ rooted forest of } G}t^{k(F)}w(F).
\end{equation}
Observe that  upon defining the weight of the arcs of $C_n$ with direction $i$ and spin $\eps$ to be $x_{i,\eps}$, the generating function $\mF_{C_n}(t;\xx)$ defined by~\eqref{def:enumeratorCn} coincides with the forest enumerator $F_{C_n}(t)$ defined by~\eqref{eq:forest-enumerator}.

Let $G=(U,A)$ and $G'=(U',A')$ be digraphs. The \emph{Cartesian product} $G\times G'$ is the digraph $H$ with vertex set $U\times U'$ and arc set obtained as follows: for every arc $a=(u,v)\in A$ and every vertex $w'\in U'$ there is an arc of $H$ from $(u,w')$ to $(v,w')$, and for every arc $a'=(u',v')\in A'$ and every vertex $w\in U$ there is an arc of $H$ from $(w,u')$ to $(w,v')$. An example of Cartesian product is given in Figure~\ref{fig:bunkbed-graph-strong} (top line). Observe that the hypercube $C_n$ is equal to the Cartesian product 
$\ds C_n=\underbrace{K_2\times \cdots \times K_2}_{n \textrm{ times}},$
where $K_2$ is the complete graph on two vertices considered as a digraph, that is, $K_2$ is the digraph with two vertices and two arcs in opposite directions joining these vertices. 



\section{Spin independence approach for the hypercube}\label{sec:bunkbed}
In this section we study the rooted forests of graphs of the form of the Cartesian product $G\times K_2$ (and more generally of certain subgraphs of the strong product $G\boxtimes K_2$). We prove an independence property for the spins of the edges of a random rooted forest of such graph: the spin of the edges in the different copies of $K_2$ are \emph{independent}. The independence property remains true if one conditions the forest to contain a given number of edges of each type, and readily gives ~\eqref{def:enumeratorCn}.


\fig{width=\linewidth}{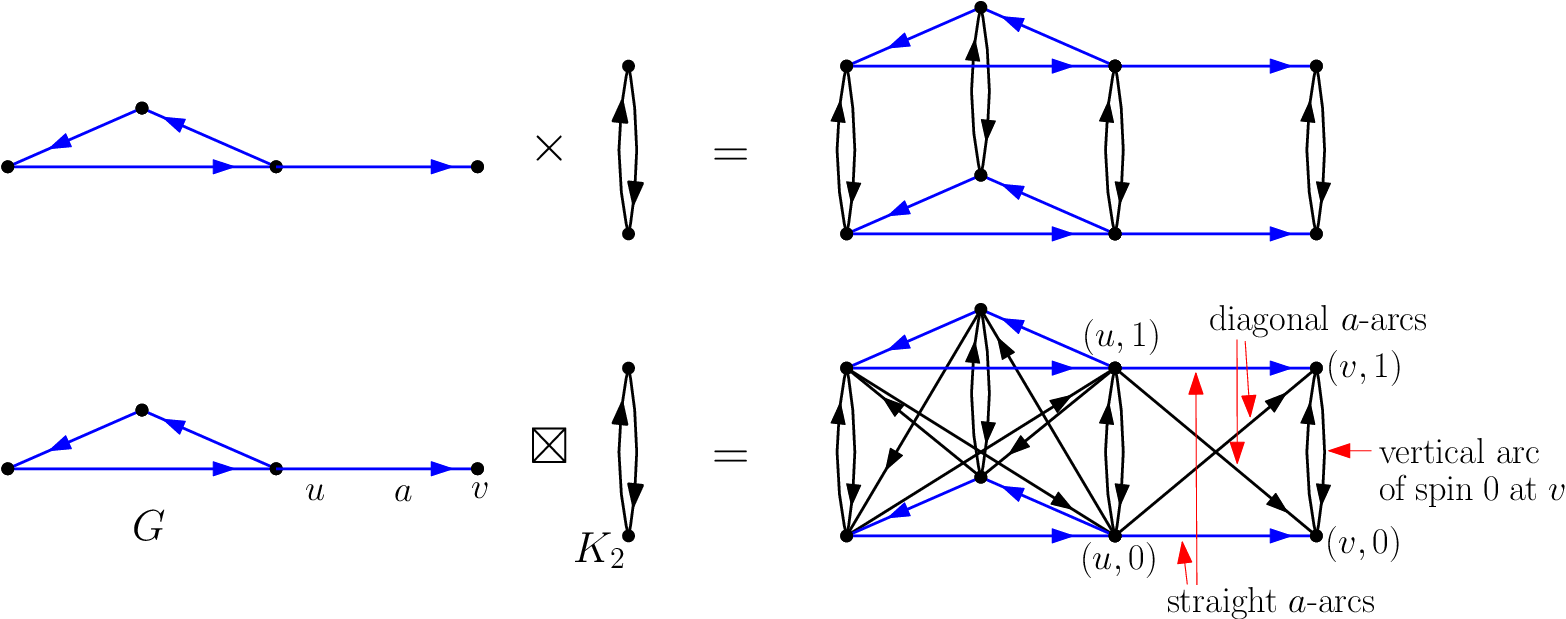}{The Cartesian product $G\times K_2$ (top line) and the strong product $G\boxtimes K_2$ (bottom line).}

We start with a few definitions. Let $G=(U,A)$ be a loopless digraph, let $K_2$ be the complete graph with vertex set $\{0,1\}$ (considered as a digraph). Recall that $G\times K_2$ denotes the Cartesian product of $G$ by $K_2$. We denote by $G\boxtimes K_2$ the \emph{strong product} of $G$ by $K_2$ which is the graph obtained from $G\times K_2$ by adding an arc from $(u,0)$ to $(v,1)$ and an arc from $(u,1)$ to $(v,0)$ for each arc $a=(u,v)$ in $A$. The graphs $G\times K_2$ and $G\boxtimes K_2$ are represented in Figure~\ref{fig:bunkbed-graph-strong}.
For $a=(u,v)\in A$ and $\eps\in\{0,1\}$, we call \emph{straight $a$-arcs} the arcs of $H=G\boxtimes K_2$ joining $(u,\eps)$ to $(v,\eps)$, and call \emph{diagonal $a$-arcs} the arcs of $H$ joining $(u,\eps)$ to $(v,1-\eps)$.
For $u\in U$ and $\epsilon\in\{0,1\}$, we call \emph{vertical arc of spin $\eps$ at $u$} the arc of $H$ from $(u,1-\eps)$ to $(u,\eps)$. Note that if $G$ has some loops then there will be both ``vertical arcs'' and ``diagonal arcs'' with the same endpoints. 
We say that two rooted forests of $H$ have the same $G$-\emph{projection} if they contain the same number of straight $a$-arcs and the same number of diagonal $a$-arcs for all $a\in A$, and moreover contain vertical arcs at the same vertices of $G$. Two rooted forests having the same projection are shown in Figure~\ref{fig:bunkbed-projection-strong}. We now state the key result of this section.

\fig{width=.7\linewidth}{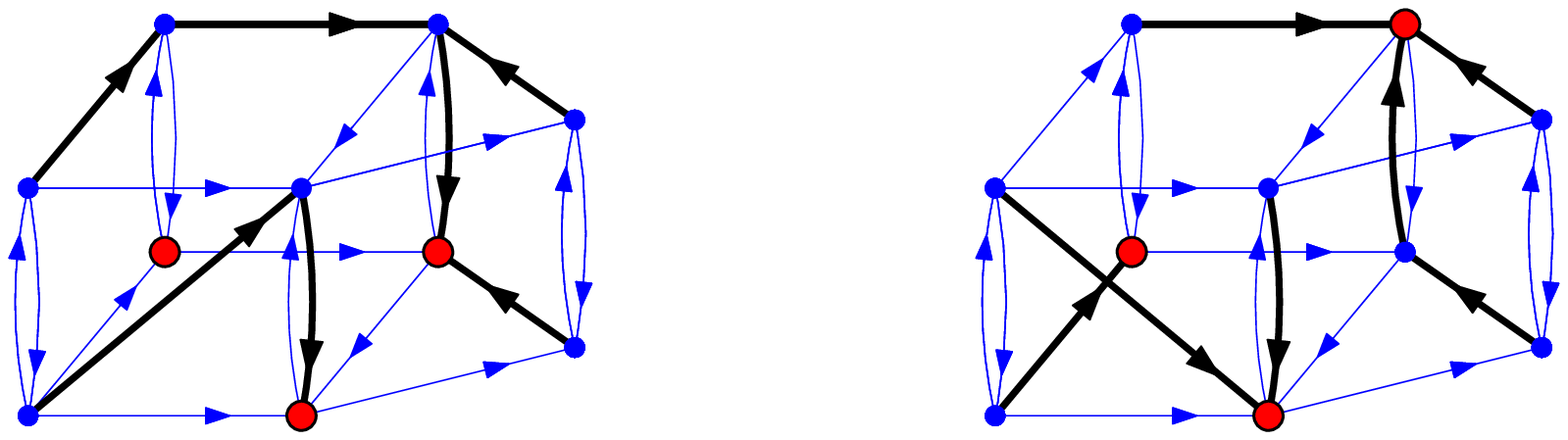}{Two rooted forests of the graph $H=G\boxtimes K_2$ having the same $G$-projection. In this picture, for the sake of readability, the diagonal arcs of $H$ not contained in the forests are not drawn. The forests are drawn in thick lines and the root vertices are represented by large dots.}

\begin{thm}\label{thm:bunkbed}
Let $G=(U,A)$ be a digraph. Let $F_0$ be a rooted forest of $H=G\boxtimes K_2$, and let $S\subseteq U$ be the set of vertices $u$ of $G$ such that $F_0$ contains a vertical arc at $u$. Let $F$ be a uniformly random rooted forest of $H$ conditioned to have the same $G$-projection as $F_0$. For all $u$ in $S$, let $\si_u\in\{0,1\}$ be the spin of the vertical arc at $u$ of the forest $F$. Then the random variables $\si_u,~u\in S$ are independent and uniformly random in~$\{0,1\}$. 
\end{thm}
Before proving Theorem~\ref{thm:bunkbed}, let us derive a few corollaries. Roughly speaking Theorem~\ref{thm:bunkbed} implies that in order to enumerate the rooted forests of $H=G\boxtimes K_2$ it is sufficient to enumerate the rooted forests without vertical arc of spin 1. The following corollary makes this statement precise.
\begin{cor}\label{cor:bunkbed-GF}
Let $G=(U,A)$ be a digraph. Let $H$ be the digraph $G\boxtimes K_2$, with arcs weighted as follows: for $\eps\in\{0,1\}$ the vertical arcs of $H$ of spin $\eps$ have weight $x_\eps$, and for $a\in A$ the straight and diagonal $a$-arcs have weight $w_a$ and $w_a'$ respectively. Then the forest enumerator $F_H(t)\equiv F_{H}(t;x_0,x_1)$ satisfies $F_{H}(t;x_0,x_1)=F_{H}(t;x_0+x_1,0)$.
\end{cor}

\begin{proof}
For a integer $v$ and tuples of integers $\mm=(m_a)_{a\in A}$ and $\nn=(n_a)_{a\in A}$, we let $\mF(v,\mm,\nn)$ be the set of rooted forests of $H$ having $v$ vertical arcs, and $m_{a}$ straight $a$-arcs and $n_a$ diagonal $a$-arcs for all $a\in A$. By Theorem~\ref{thm:bunkbed}, the number of vertical arcs of spin 0 in a uniformly random forest $F$ in $\mF(v,\mm,\nn)$ has a binomial distribution with parameter $(v,1/2)$. Hence, 
$$\sum_{F\in \mF(v,\mm,\nn)}x_0^{\#\textrm{vertical arcs of spin }0}\,x_1^{\#\textrm{vertical arcs of spin }1}=|\mF_{v,\mm,\nn}|\left(\frac{x_0}{2}+\frac{x_1}{2}\right)^v.$$
Thus, 
$$F_{H}(t;x_0,x_1)\equiv \sum_{\mm,\nn,v}\prod_{a\in A}w_a^{m_a}{w_{a}'}^{n_a}\sum_{F\in \mF(v,\mm,\nn)}x_0^{\#\textrm{vertical arcs of spin }0}\,x_1^{\#\textrm{vertical arcs of spin }1}$$
is unchanged when replacing $(x_0,x_1)$ by $(x_0+x_1,0)$.
\end{proof} 
In the next two corollaries, we focus on the forests of the Cartesian products $G\times K_2$,  which are simply the forests of $G\boxtimes K_2$ without diagonal arcs.

\begin{cor}\label{cor:K2-induction}
Let $G=(U,A)$ be a weighted digraph with the weight of an arc $a\in A$ denoted by $w_a$. 
Let $H$ be the digraph $G\times K_2$ with arcs weighted as follows: for $\eps\in\{0,1\}$ the vertical arcs of spin $\eps$ of $H$ have weight $x_\eps$, and for $a\in A$ the (straight) $a$-arcs of $H$ have weight $w_a$. Then the forest enumerators of $G$ and $H$ are related by 
$$F_H(t)=F_G(t)F_G(t+x_0+x_1).$$
\end{cor}

\begin{proof}
Let us denote $F_H(t)=F_H(t;x_0,x_1)$ in order to make explicit the dependence in the variables $x_0,x_1$. By applying Corollary~\ref{cor:bunkbed-GF} (in the special case where the weights of diagonal arcs are 0), we get $F_H(t;x_0,x_1)=F_H(t;x_0+x_1,0)$, so it only remains to prove that 
$$F_H(t;x_0,0)=F_G(t)F_G(t+x_0).$$
Now, by definition 
$F_H(t;x_0,0)=\sum_{F\in \mF'}w(F)$,
where $\mF'$ is the set of rooted forests without vertical arc of spin 1. For $\eps\in\{0,1\}$, let $G_\eps$ be the subgraph of $H$ isomorphic to $G$ induced by the vertices of the form $(u,\eps),u\in U$. Clearly, any rooted forest in $\mF'$ is obtained by 
\begin{compactitem}
\item[(i)] choosing a rooted forest $F_0$ of $G_0$, 
\item[(ii)] choosing a rooted forest $F_1$ of $G_1$, and then choosing for each root vertex of $F_1$ whether to add a vertical arc (of spin 0) out of this vertex, 
\end{compactitem}
and any choice (i), (ii) gives a rooted forest in $\mF'$ (since it is impossible to create cycles by adding the vertical arcs). Moreover $F_G(t)$ is the generating function of all the possible choices for (i), while $F_G(t+x_0)$ is the generating function of all the possible choices for (ii). This completes the proof. 
\end{proof}

As mentioned earlier the hypercube $C_n$ is equal to $K_2\times \cdots \times K_2$ and we now use Theorem~\ref{thm:bunkbed} to prove~\eqref{eq:spanning-cube-2}.

\begin{cor}\label{cor:cube}
The $n$-dimensional hypercube $C_n$ with weight $x_{i,\eps}$ for the arcs having direction $i$ and spin $\eps$ has forest enumerator
$$F_{C_n}(t;\xx)=\prod_{S\subseteq[n]}\left(t+\sum_{i\in S} x_{i,0}+x_{i,1}\right).$$ 

\end{cor}

\begin{proof}
Corollary~\ref{cor:cube} follows from Corollary~\ref{cor:K2-induction} by induction on $n$. Below we give a slightly more direct proof. First observe that by Corollary \ref{cor:bunkbed-GF},  the forest enumerator $F_{C_n}(t;\xx)$ is unchanged by replacing for all $i\in[n]$ the variables $x_{i,0}$ and $x_{i,1}$ respectively by $x_{i,0}+x_{i,1}$ and $0$. Hence it only remains to prove 
\begin{equation}\label{eq:no1}
F_{C_n}(t;x_{1,0},0,\ldots,x_{n,0},0)=\prod_{S\subseteq[n]}\left(t+\sum_{i\in S} x_{i,0}\right).
\end{equation}
By definition, $\ds F_{C_n}(t;x_{1,0},0,\ldots,x_{n,0},0)=\sum_{F\in\mF'}w(F),$
where $\mF'$ is the set of rooted forests of $C_n$ without arc of spin 1. A forest in $\mF'$ represented in Figure~\ref{fig:cube-with-diagonals}. 
Such a forest is obtained by choosing for each vertex $v=(v_1,\ldots,v_n)\in\{0,1\}^n$ either to make this vertex a root vertex (this contributes weight $t$) or to make it a vertex with outgoing arc (of spin 0) in direction $i$ for $i$ in the subset $S_v=\{i\in[n],~v_i=1\}$ (this contributes weight $\sum_{i\in S_v} x_{i,0}$). Since any such choice leads to a distinct rooted forest in $\mF'$, we get~\eqref{eq:no1}.
\end{proof}

The rest of this section is devoted to the proof of Theorem~\ref{thm:bunkbed}.

\begin{proof}[Proof of Theorem~\ref{thm:bunkbed}] Let $\mF_0$ be the set of rooted forests of $H=G\boxtimes K_2$ having the same $G$-projection as $F_0$. The rooted forest $F$ is chosen uniformly in $\mF_0$ and we want to prove that the spins of its vertical arcs are uniformly random and independent. We will prove this property by induction on the number $n$ of vertices of $G$. The property is obvious for $n=1$. We now suppose that it holds for any graph $G'$ with less than $n$ vertices, and we want to prove the property for $G$. 

Let $\al$ be the number of non-vertical arcs of $F_0$. Let $\be=|S|$ be the number of vertices $u$ of $G$ such that $F_0$ contains a vertical arc at $u$, and let $\ga=|U\setminus S|=n-\be$ be the number of other vertices of $G$. Since the forest $F_0$ has $\al+\be$ arcs and $2n$ vertices, one gets $\al+\be<2n$, hence $\al<\be+2\ga$. 
Thus there exists either 
\begin{compactitem}
\item[(a)] a vertex $u\in S$ such that $F_0$ contains no $a$-arc with $a\in A$ directed toward $u$, 
\item[(b)] or a vertex $u\in U\setminus S$ such that $F_0$ contains at most one $a$-arc with $a\in A$ directed toward $u$. 
\end{compactitem}
Cases (a) and (b) are illustrated in Figure~\ref{fig:proof-bunkbed-strong}. In both cases we will apply the induction hypothesis on graphs obtained from $G$ by deleting the vertex $u$.
\fig{width=.9\linewidth}{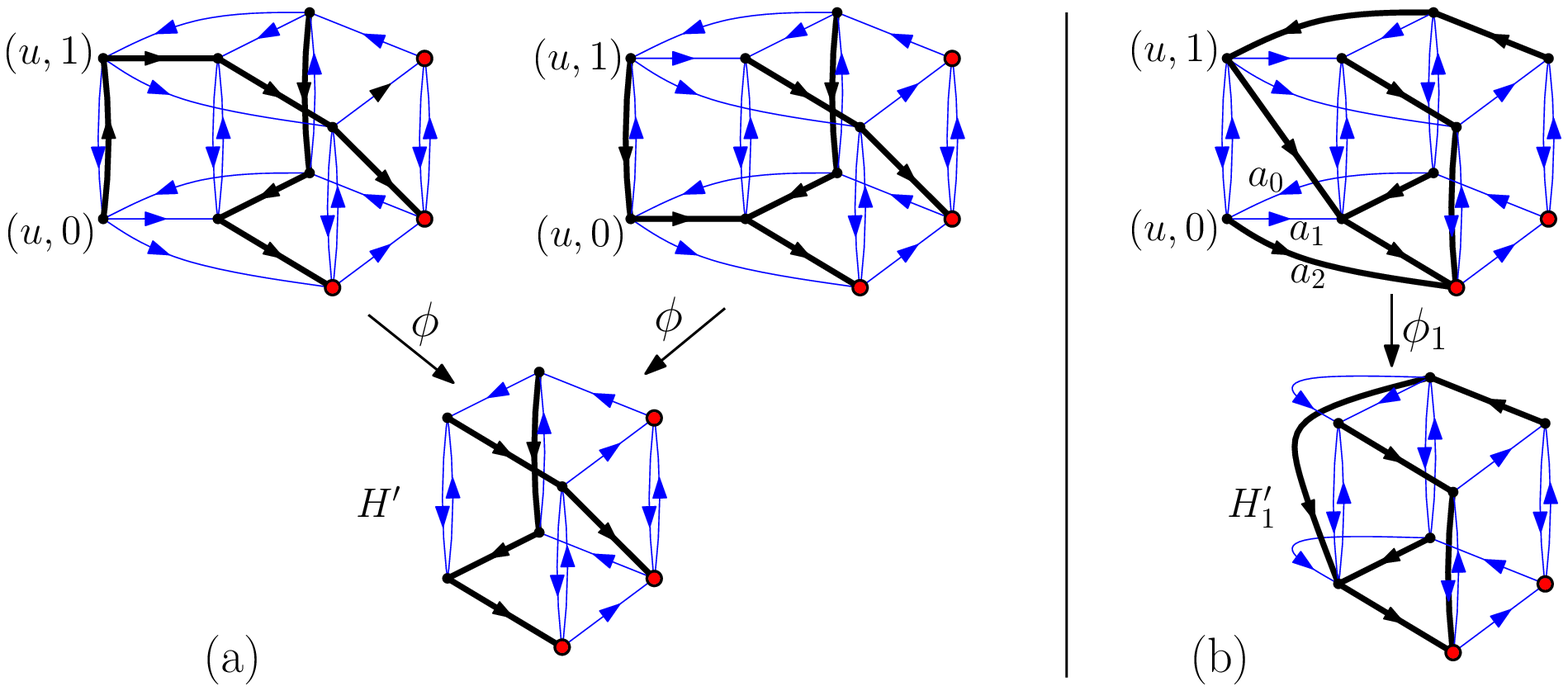}{Cases (a) and (b) of the inductive proof of Theorem~\ref{thm:bunkbed}. In this picture, for the sake of readability, the diagonal arcs of $H=G\boxtimes K_2$ not contained in the forests are not drawn.}

We first consider the case (a). Let $G'$ be the digraph obtained from $G$ by deleting the vertex $u$ and the incident arcs, and let $H'=G'\boxtimes K_2$. For a rooted forest $T\in\mF_0$ we denote by $\phi(T)$ the rooted forest of $H'$ obtained from $T$ by deleting the vertices $(u,0)$ and $(u,1)$ and the incident arcs; see Figure~\ref{fig:proof-bunkbed-strong}(a). 
Let $F_0'=\phi(F_0)$ and let $\mF_0'$ be the set of rooted forests of $H'$ having the same $G'$-projection as $F_0'$. 
It is easy to see that any rooted forest $T'\in\mF_0'$ has exactly two preimages in $\mF_0$ by the mapping $\phi$: one preimage having a vertical arc of spin 0 at $u$ and the other having a vertical arc of spin 1 at $u$. Hence, if $F$ is uniformly random in $\mF_0$ then $F'=\phi(F)$ is uniformly random in $\mF_0'$. Thus, by the induction hypothesis, the vertical arcs of $F'$ are uniformly random and independent. Moreover the spin $\si_u$ of the vertical arc of $F$ at $u$ is uniformly random and independent of the forest $F'=\phi(F)$. Thus, the spins of all the vertical arcs of $F$ are uniformly random and independent, as wanted.

We now consider the case (b). There is at most one arc $a\in A$ directed toward $u$ such that $F_0$ contains a $a$-arc, and at most two arcs $a'\in A$ directed away from $u$ such that $F_0$ contains an $a'$-arc. Considering all the possibilities is a bit tedious (but not hard), so we shall only treat the most interesting case in detail: we suppose that there is an arc $a_0\in A$ directed toward $u$ and two distinct arcs $a_1,a_2\in A$ directed away from $u$ such that $F_0$ contains an $a_i$-arc for all $i\in\{0,1,2\}$. This situation is represented in Figure~\ref{fig:proof-bunkbed-strong}(b); in that figure the $a_1$-arc of $F$ is a diagonal arc and the $a_0$-arc and $a_2$-arc of $F$ are straight arcs. We partition $\mF_0$ into two subsets $\mF_0=\mF_1\uplus\mF_2$, where for $i\in\{1,2\}$, $\mF_i$ is the set of rooted forests $T\in\mF_0$ such that the $a_0$-arc and $a_i$-arc of $T$ are incident to the same vertex of $H$. It is sufficient to prove that for $i\in\{1,2\}$, if $\mF_i\neq\emptyset$ and $F_i$ is a uniformly random rooted forest in $\mF_i$ then the spins of the vertical arcs of $F_i$ are uniformly random and independent. 

Let $i\in\{1,2\}$ be such that $\mF_i$ is not empty. Let $G_i'$ be the digraph obtained from $G$ by merging the arcs $a_0$ and $a_i$ into a single arc $b$ (going from the origin of $a_0$ to the end of $a_1$) and then deleting the vertex $u$ and all the incident arcs, and let $H_i'=G_i'\boxtimes K_2$. For $T\in \mF_i$ we denote by $\phi_i(T)$ the forest of $H_i'$ obtained from $T$ by merging the $a_0$-arc and the $a_i$-arc into a single arc (the arc created will be a straight $b$-arc if $a_0$ and $a_i$ are both straight or both diagonal, and a diagonal $b$-arc otherwise) and then deleting the vertices $(u,0)$ and $(u,1)$ and the incident arcs . Let $F_{i,0}$ be a rooted forest in $\mF_i$, let $F_{i,0}'=\phi_i(F_{i,0})$, and let $\mF_i'$ be the set of rooted forest of $H_i'$ having the same $G_i'$-projection as $F_{i,0}'$. It is easy to see that $\phi_i$ is a bijection between $\mF_i$ and $\mF_i'$. Thus, if $F_i$ is uniformly random in $\mF_i$ then $F_i'=\phi_i(F_i')$ is uniformly random in $\mF_i'$. Hence, by the induction hypothesis, the spins of the vertical arcs of $F_i$ (which are the same as the spins of the vertical arcs of $F_i'$) are uniformly random and independent. This completes the proof.
\end{proof}
\smallskip

\section{Root identification approach for the hypercube with diagonals}\label{sec:withdiago}
In this section we consider the graph $D_n$ obtained from the hypercube $C_n$ by adding a \emph{diagonal arc} from each vertex $v=(v_1,\ldots,v_n)\in\{0,1\}$ to its \emph{antipodal vertex} $v'=(1-v_1,\ldots,1-v_n)$. The graph $D_3$ is represented in Figure~\ref{fig:cube-with-diagonals}. Let $F_{D_n}(t;\xx,y)$ be the forest enumerator of $D_n$ defined by 
\begin{equation}\label{eq:FDn}
F_{D_n}(t;\xx,y)=\sum_{F\textrm{ rooted forest of } D_n}\,\prod_{a\in F}w(a),
\end{equation}
where the weight $w(a)$ of a diagonal arc is $y$ and the weight of a non-diagonal arc of direction $i$ and spin $\eps$ is $x_{i,\eps}$. The main result of this section is the following product formula for $F_{D_n}(t;\xx,y)$. 
\begin{thm}\label{thm:with-diago}
The forest enumerator of the hypercube with diagonals defined by~\eqref{eq:FDn} equals
\begin{equation}\label{eq:with-diago}
F_{D_n}(t;\xx,y)=\prod_{S\subseteq [n]}\left(t+2y\cdot\textbf{1}_{|S| \textrm{ odd}}+\sum_{i\in S}x_{i,0}+x_{i,1}\right).
\end{equation}
\end{thm}
Observe that the forest enumerator of the hypercube is $F_{C_n}(t;\xx)=F_{D_n}(t;\xx,0)$, hence Theorem~\ref{thm:with-diago} gives a generalization of~\eqref{eq:spanning-cube-2}. The rest of this section is devoted to the proof of Theorem~\ref{thm:with-diago}. 
The proof below uses Theorem~\ref{thm:bunkbed} as one of its ingredients\footnote{Using Corollary \ref{cor:bunkbed-GF} it is actually sufficient to prove \eqref{eq:with-diago} in the special case where $x_{1,1}=x_{2,1}=\ldots=x_{n,1}=0$, but we have not found a more direct proof of this special case.}.
However the case $y=0$ corresponding to~\eqref{eq:spanning-cube-2} does not require Theorem~\ref{thm:bunkbed} hence we obtain an independent combinatorial proof of this formula.

It is clear from the definitions that $F_{D_n}(t;\xx,y)$ is a monic polynomial in $t$ of degree~$2^n$. Thus, in order to prove Theorem~\ref{thm:with-diago} it suffices to show that for all $S\subseteq [n]$,
\begin{equation}\label{eq:tobeproved}
F_{D_n}(-2y\cdot\textbf{1}_{|S| \textrm{ odd}}-\sum_{i\in S}x_{i,0}+x_{i,1};\xx,y)=0.
\end{equation}
We now fix a subset $S\subseteq [n]$ and establish~\eqref{eq:tobeproved} by exhibiting some ``killing involutions''. We denote $S'=S$ if $S$ is even and $S'=S\cup \{0\}$ if $S$ is odd. We also say that the diagonal arcs of $D_n$ have \emph{direction 0}.
A rooted forest of $D_n$ is $S$-\emph{labeled} if every root vertex has a \emph{label} in $S'$. For a $S$-labeled forest $F$, we denote 
$$w(F)=\prod_{a \emph{ arc of }F}w(a)\times \prod_{r\textrm{ root vertex of } F}\ow(r),$$
where $\ow(r)=-2y$ if the root vertex $r$ is labeled 0, and $\ow(r)=-x_{i,0}-x_{i,1}$ if $r$ is labeled $i>0$. With this notation we immediately get
$$
F_{D_n}(-2y\cdot\textbf{1}_{|S| \textrm{ odd}}-\sum_{i\in S}x_{i,0}+x_{i,1};\xx,y)=\sum_{F\in \mF_S}w(F),
$$
and it remains to prove that 
\begin{equation}\label{eq:tobeproved2}
\sum_{F\in \mF_S}w(F)=0.
\end{equation}

Let $\mC$ be the set of subgraphs of $D_n$ such that every vertex of $D_n$ is incident to exactly one outgoing arc. Any element $C\in\mC$ is made of a some disjoint directed cycles together with directed trees rooted on the vertices of the cycles. 
For a labeled rooted forest $F\in \mF_S$, we denote by $\ov{F}$ the subgraph in $\mC$ obtained from $F$ by adding the arc of direction $i$ going out of each root vertex labeled $i$ for all $i\in \{0,\ldots,n\}$ (with the convention that diagonal arcs have direction 0). We get
$\ds \sum_{F\in \mF_S}w(F)=\sum_{C\in\mC}~\sum_{F\in \mF_S,~\ov{F}=C}w(F),$
and now proceed to compute $\ds \sum_{F,\,\ov{F}=C}w(F)$ for a given subgraph $C\in\mC$.

Let $C\in \mC$ and let $C^{(1)},\ldots,C^{(k)}$ be the directed cycles of $C$, and let $C^{(0)}$ be the set arcs of $C$ which are not in cycles. For all $j\in[k]$, we denote by $C_S^{(j)}$ the set of arcs in $C^{(j)}$ having their direction in $S'$, and we denote $\ov{C}^{(j)}_{S}=C^{(j)}\setminus C^{(j)}_{S}$. 
The forests $F$ such that $\ov{F}=C$ are obtained from $C$ by removing an arbitrary subset of arcs in $C^{(0)}_S$, and by removing a non-empty subset of arcs in $C^{(j)}_S$ for all $j=1,\ldots,k$. Moreover, if an arc $a\in C$ is not removed then its contribution to the weight $w(F)$ is $w(a)$, while if $a$ is removed then its contribution to $w(F)$ is $\ow(a)=-2y$ if $a$ is a diagonal edge and $\ow(a)=-x_{i,0}-x_{i,1}$ if it is a non-diagonal edge of direction $i$ and spin $\eps$. Thus 
$$\sum_{F\in\mF_S,~\ov{F}=C}\!\!\!w(F)=\prod_{a\in \ov{C}_{S}^{(0)}}\!\!w(a)\prod_{a\in C_S^{(0)}}\!\!\tw(a)\times\prod_{j=1}^k\prod_{a\in \ov{C}_{S}^{(j)}}\!\!w(a)\left(\prod_{a\in C^{(j)}_S}\!\!\tw(a)-\prod_{a\in C_{S}^{(j)}}\!\!w(a)\right),$$
where $\tw(a)=w(a)+\ow(a)$.
Now we claim that for all $j\in[k]$ the number of arcs in $C^{(j)}_S$ is even. For this purpose, consider for each vertex $v\in\{0,1\}^n$ the quantity $v_S=\sum_{i\in S}v_i$. The parity of $v_s$ is changing along arcs having direction in $S'$ and not changing along the other arcs. Hence for all $j\in[k]$ the number of arcs with direction in $S'$ in the cycle $C^{(j)}$ is even. Therefore, 
\begin{equation}\label{eq:tobekilled}
\sum_{F\in\mF_S,~\ov{F}=C}\!\!\!\!w(F)=\prod_{a\in \ov{C}_{S}^{(0)}}\!\!\!w(a)\prod_{a\in C_S^{(0)}}\!\!\!\tw(a)\times\prod_{j=1}^k\prod_{a\in \ov{C}_{S}^{(j)}}\!\!w(a)\left(\prod_{a\in C^{(j)}_S}\!\!\!\hw(a)-\prod_{a\in C_{S}^{(j)}}\!\!\!w(a)\right),
\end{equation}
where $\hw(a)=-\tw(a)$, that is, $\hw(a)=y$ if $a$ is a diagonal edge, and $\hw(a)=x_{i,1-\eps}$ if $a$ is a non-diagonal edge of direction $i\in S$ and spin $\eps$. 

Let us now briefly consider the case $y=0$, which leads to an independent proof of~\eqref{eq:spanning-cube-2}. In the case $y=0$ the right-hand side of~\eqref{eq:tobekilled} is clearly 0 unless $C$ has no diagonal arc. Moreover, if $C$ has no diagonal arc then for all $j\in[k]$, $i\in[n]$ the cycle $C^{(j)}$ contains as many arcs with direction $i$ and spin 0 as arcs with direction $i$ and spin 1. Hence in this case, 
$\ds \prod_{a\in C^{(j)}_S}\!\!\hw(a)=\prod_{a\in C_{S}^{(j)}}\!\!w(a),$
so that $\ds \sum_{F\in\mF_S,~\ov{F}=C}w(F)$ is always~0. Thus, we have proved~\eqref{eq:tobeproved2} and therefore~\eqref{eq:with-diago} in the case $y=0$, which is precisely~\eqref{eq:spanning-cube-2}.

We now resume our analysis in the case $y\neq 0$. It is not true that the right-hand side of~\eqref{eq:tobekilled} is~0 in general. However in the case where $x_{i,0}=x_{i,1}$ for all $i\in S$ one has $\hw(a)=w(a)$ for all arc $a$ having direction in $S'$, hence $\sum_{F\in\mF_S,~\ov{F}=C}w(F)=0$.
This gives~\eqref{eq:tobeproved2} and therefore~\eqref{eq:with-diago} in the case where $x_{i,0}=x_{i,1}$ for all $i\in [n]$. Equivalently,
$$F_{D_n}(t;\xx',y)=\prod_{S\subseteq [n]}\left(t+2y\cdot\textbf{1}_{|S| \textrm{ odd}}+\sum_{i\in S}x_{i,0}+x_{i,1}\right),$$
for $\ds \xx'=\left(\frac{x_{1,0}+x_{1,1}}{2},\frac{x_{1,0}+x_{1,1}}{2},\ldots,\frac{x_{n,0}+x_{n,1}}{2},\frac{x_{n,0}+x_{n,1}}{2}\right)$.
 We now combine this result with Corollary~\ref{cor:bunkbed-GF}.

First observe that $D_n$ is obtained from the digraph $D_{n-1}\boxtimes K_2$ by removing its straight $a$-arcs for every diagonal arc $a$ of $D_{n-1}$ and removing its diagonal $a'$-arcs for every non-diagonal arc $a'$ of $D_{n-1}$. Therefore, Corollary~\ref{cor:bunkbed-GF} implies that $F_{D_n}(t;\xx,y)$ is unchanged by replacing $(x_{n,0},x_{n,1})$ by $(x_{n,0}+x_{n,1},0)$ or by $\ds (\frac{x_{n,0}+x_{n,1}}{2},\frac{x_{n,0}+x_{n,1}}{2})$. By symmetry, a similar result is true for every direction $i\in[n]$. Therefore $F_{D_n}(t;\xx,y)=F_{D_n}(t;\xx',y)$. This completes the proof of Theorem~\ref{thm:with-diago}.\hfill $\square$

\section{Matrix-tree approach for products of complete graphs}\label{sec:matrix-tree}
In the previous sections we gave combinatorial proofs of formula~\eqref{eq:spanning-cube-2} for the forest enumerator of the hypercube. In this section we instead use the good old matrix-tree approach to establish some generalizations for Cartesian products of complete graphs. 

Let $G=(U,A)$ and $G'=(U',A')$ be weighted digraphs and let $w_a$ be the weight of any arc $a$ in $A\cup A'$. The \emph{weighted Cartesian product} of $G$ and $G'$, is the digraph $H=G\times G'$ where for any  arc $a=(u,v)\in A\cup A'$, the arcs of $H$ corresponding to $a$ (if $a\in A$, these are arcs going from $(u,w)$ to $(v,w)$ for $w\in U'$, while if $a\in A'$ these are arcs going from $(w,u)$ to $(w,v)$ for $w\in U$) have weight $w_a$. 
Observe that the weighted Cartesian product $K_2\times\cdots \times K_2$ of $n$ copies of $K_2$ with the $i$th copy having arc weights $x_{i,0}$ and $x_{i,1}$ is equal to the hypercube with weight $x_{i,\eps}$ for arcs having direction $i$ and spin $\eps$. The following proposition will be proved by combining the matrix-tree theorem with a classical result about the eigenvalues of the Laplacian of a Cartesian product of graphs (see e.g. \cite{Fiedler:eigenvalues-Cartesian}).
\begin{prop}\label{prop:Laplacian}
Let $G$, $G'$ be weighted digraph with respectively $p$ and $q$ vertices. 
Let $F_G(t)$, $F_{G'}(t)$ be the forest enumerators of $G$ and $G'$ as defined by~\eqref{eq:forest-enumerator}. Let $\la_1,\ldots,\la_p$ and $\la_1',\ldots,\la_{q}'$ be the roots (appearing with multiplicity) of $F_G(t)$ and $F_{G'}(t)$ considered as polynomials in $t$ (the roots are taken in the splitting field of these polynomials). Then the forest enumerator of the weighted Cartesian product $H=G\times G'$ is
$$F_H(t)=\prod_{i\in [p],j\in[q]}(t+\la_i+\la_j')= \prod_{j\in [q]}F_{G}(t+\la_j')=\prod_{i\in [p]}F_{G'}(t+\la_i).$$
\end{prop}

Observe that Corollary~\ref{cor:K2-induction} is a special case of Proposition~\ref{prop:Laplacian} corresponding to $G'=K_2$ (with weight $x_0$ and $x_1$ on the edges of $K_2$). 
Before proving Proposition~\ref{prop:Laplacian}, we explore its consequences for products of complete graphs. We first recall a classical result about the forest enumerator of complete graphs. Let $K_p$ be the complete graph with vertex set $[p]$ (considered as a digraph with $p(p-1)$ arcs). If for all $j \in[p]$ the arcs of $K_p$ directed toward the vertex $j$ are weighted by $x_j$, then the forest enumerator of $K_p$ is 
$$F_{K_p}(t)=t\,(t+x_1+\cdots+x_{p})^{p-1}.$$
This classical result, often attributed to Cayley, has many beautiful proofs~\cite[Chapter 26]{AZ}. Since the roots of $F_{K_p}(t)$ are known explicitly for all $p$,  Proposition~\ref{prop:Laplacian} immediately gives the following result (by induction on $n$).
\begin{corollary}\label{cor:product-complete}
Let $p_1,\ldots,p_n$ be positive integers and $K_{p_1},\ldots,K_{p_n}$ be complete graphs with $p_1,\ldots,p_n$ vertices respectively. For all $i\in[n]$, let the $i$th complete graph $K_{p_i}$ be weighted by assigning a weight $x_{i,\eps}$ to every arc going toward the vertex $\eps$ for all $\eps\in [p_i]$. Then the weighted Cartesian product $K_{p_1}\times \cdots \times K_{p_n}$ has the following forest enumerator 
\begin{equation}\label{eq:product-complete}
F_{K_{p_1}\times \cdots \times K_{p_n}}(t)=\prod_{(v_1,\ldots,v_n)\in[p_1]\times \cdots \times[p_n]}(t+\sum_{i,~v_i\neq 1}x_{i,1}+\ldots+x_{i,p_i}).
\end{equation}
\end{corollary}
Corollary \ref{cor:product-complete} is closely related to a formula established by Martin and Reiner in \cite{Martin-Reiner:spanning-cube} using a method similar to ours. Indeed \cite[Theorem 1]{Martin-Reiner:spanning-cube} is equivalent (up to easy algebraic manipulations) to the special case  $x_{i,1}=x_{i,2}=\ldots=x_{i,p_i}$ of \eqref{eq:product-complete}.
Observe also that formula~\eqref{eq:spanning-cube-2} for the hypercube corresponds to the case $p_1=\ldots=p_n=2$ of Corollary~\ref{cor:product-complete} (upon identifying the subsets of $[n]$ with the elements of $[2]^n$).

The rest of this section is devoted to the proof of Proposition~\ref{prop:Laplacian}. We first recall the matrix-tree theorem.
Let $G$ be a simple weighted digraph with vertex set $[n]$. For two vertices $i,j\in [n]$ we define $w_{i,j}$ to be the weight of the arc from vertex $i$ to vertex $j$ if there is such an arc, and to be 0 otherwise. The \emph{Laplacian} of $G$, denoted $L(G)$, is the $n\times n$ matrix whose entry at position $(i,j)\in[n]^2$ is equal to $-w_{i,j}$ if $i\neq j$ and to $\sum_{k=1}^nw_{i,k}$ otherwise.
We now recall the (directed, weighted, forest version of) the matrix-tree theorem\footnote{Our weights $w_{i,j}$ are arbitrary indeterminates as authorized by the combinatorial proofs of the matrix-tree theorem (see e.g.~\cite{Zeilberger:combinatorial-matrix}).} which gives the forest-enumerator of $G$ as a determinant:
\begin{equation}\label{eq:matrix-tree}
F_G(t)\equiv\sum_{F\textrm{ rooted forest of } G}t^{k(F)}w(F)=\det\left(L(G)+t\cdot \Id_n\right),
\end{equation}
where $\Id_n$ denotes the identity matrix of dimension $n\times n$. In other words, for any weighted digraph $G$ the roots of the forest enumerator $F_G(t)$ are the opposite of the eigenvalues of the Laplacian $L(G)$. In order to complete the proof Proposition \ref{prop:Laplacian}, it now suffices to combine this fact with the following result of Fiedler \cite{Fiedler:eigenvalues-Cartesian} (Fiedler actually only considered undirected unweighted graph, but the proof allows for arbitrary weights).
\begin{lemma}[\cite{Fiedler:eigenvalues-Cartesian}]\label{lem:eigenvalues-add-up}
If $G$, $G'$ and $H$ are as in Proposition \ref{prop:Laplacian} and the eigenvalues (taken with multiplicities) of the Laplacians $L(G)$ and $L(G')$ are $\la_1,\ldots,\la_p$ and $\la_1',\ldots,\la_q'$ respectively, then the eigenvalues of $L(H)$ (taken with multiplicities) are $(\la_i+\la'_j)_{i\in[p],j\in[q]}$.
\end{lemma}
\begin{proof}[Sketch of proof of Lemma \ref{lem:eigenvalues-add-up}] The Laplacians of $G$, $G'$ and $H$ are related by
$$L(H)=L(G)\otimes \Id_{q}+\Id_p\otimes L(G'),$$
where ``$\otimes$'' represents the  \emph{Kronecker product} of matrices. 
Moreover, if $M,N$ are any matrices of dimension $p\times p$ and $q\times q$ respectively, with eigenvalues  $\la_1,\ldots,\la_p$ and $\la_1',\ldots,\la_q'$, then the eigenvalues of the matrix $L=M\otimes \Id_q+\Id_p\otimes N$  are $(\la_i+\la'_j)_{i\in[p],j\in[q]}$. 
Indeed, there exists invertible matrices $P,Q$ (with entries in the splitting field of the polynomial $\det(M+t\cdot\Id_p)\det(N+t\cdot\Id_q)$) such that the matrices $M':= P^{-1}MP$ and $N':= Q^{-1}NQ$ are both upper triangular with diagonal elements $\la_1,\ldots,\la_p$ and $\la_1',\ldots,\la_q'$ respectively. And it is easily seen that $$(P\otimes Q)^{-1}\cdot\left(M\otimes \Id_q+\Id_p\otimes N\right)\cdot(P\otimes Q)=M'\otimes \Id_p+\Id_q\otimes N'$$
is a upper triangular matrix with diagonal elements $(\la_i+\la'_j)_{i\in[p],j\in[q]}$.  
This completes the proof of Lemma \ref{lem:eigenvalues-add-up} and Proposition \ref{prop:Laplacian}.
\end{proof}

\section{Additional remarks and conjectures}\label{sec:conclusion}
In this section we first give a formula for the enumerator of the spanning trees of the hypercube rooted at a given vertex, and explain its relation with a formula by Martin and Reiner~\cite{Martin-Reiner:spanning-cube}. Then we mention a consequence of Theorem~\ref{thm:bunkbed} and conjecture a generalization of this theorem.\\ 

\noindent \textbf{Unrooted spanning trees of the hypercube and relation with~\cite[Theorem 3]{Martin-Reiner:spanning-cube}}.\\
For a vertex $v=(v_1,v_2,\ldots,v_n)\in\{0,1\}^n$ of the hypercube we denote by $\mT_v$ the set of spanning trees of $C_n$ rooted at the vertex $v$, and we denote
$$T_{C_n,v}(\xx)=\sum_{T\in\mT_v}\,\prod_{a\in T}x_{\dir(a),\spin(a)}.$$ 
Observe that if $u=(u_1,\ldots,u_n)\in\{0,1\}^n$ is another vertex of $C_n$, then 
$$T_{C_n,u}(\xx)=\left(\prod_{i=1}^n\frac{x_{i,u_i}}{x_{i,v_i}}\right) T_{C_n,v}(\xx),$$
since changing the root of a spanning tree from $v$ to $u$ changes the number of arcs of direction $i$ and spin $1$ (resp. 0) by $u_i-v_i$ (resp. $v_i-u_i$). Combining this relation with
$$\sum_{v\in c_n}T_{C_n,v}(\xx)= [t]F_{C_n}(t;\xx)=\prod_{S\subseteq [n],S\neq\emptyset}\,\sum_{i\in S}(x_{i,0}+x_{i,1}),$$
gives 
\begin{equation}\label{eq:rooted-at-v}
T_{C_n,v}(\xx)=\left(\prod_{i=1}^nx_{i,v_i}\right)\times\left(\prod_{S\subseteq [n],|S|\geq 2}\,\sum_{i\in S}(x_{i,0}+x_{i,1})\right).
\end{equation}
We now establish the equivalence of \eqref{eq:rooted-at-v} with~\cite[Theorem 3]{Martin-Reiner:spanning-cube}. For $i\in[n]$, $\eps\in\{0,1\}$ and $T$ an unrooted spanning tree of $C_n$, we denote by $\deg_{i,\eps}(T)$ the sum of the degrees in $T$ of all the vertices of $C_n$ having their $i$th coordinate equal to $\eps$. We then consider 
$$S_{C_n}(\qq,\yy)=\sum_{T}\left(\prod_{e\in T}q_{\dir(e)}\right)\times\left(\prod_{i=1}^ny_{i,0}^{\deg_{i,0}(T)}y_{i,1}^{\deg_{i,1}(T)}\right),$$
where the sum is over the unrooted spanning trees of $C_n$. Let $T$ be an unrooted spanning tree of $C_n$ and let $T'$ be the rooted tree obtained by choosing $\rho=(0,0,\ldots,0)$ as the root vertex. It is easy to see that  for all $i\in[n],\eps\in\{0,1\}$
$$\deg_{i,\eps}(T)=2^n-2\cdot \textbf{1}_{\eps=0}+n_{i,\eps}(T')-n_{i,1-\eps}(T'),$$ 
where $n_{i,\eps}(T')$ is the number of arcs of $T'$ with direction $i$ and spin $\eps$. Therefore 
$$S_{C_n}(\qq,\yy)=\left(\prod_{i=1}^ny_{i,0}^{2^n}\,y_{i,1}^{2^n}\right)\times \frac{T_{C_n,\rho}(\xx)}{\prod_{i=1}^ny_{i,0}^2},$$
with $x_{i,\eps}=q_i\, y_{i,\eps}/y_{i,1-\eps}$. Using~\eqref{eq:rooted-at-v} then gives the following result obtained by Martin and Reiner in~\cite[Theorem 3]{Martin-Reiner:spanning-cube} using a matrix-tree approach:
$$S_{C_n}(\qq,\yy)=\left(\prod_{i=1}^nq_i\,y_{i,0}^{2^n-1}\,y_{i,1}^{2^n-1}\right)\times\left(\prod_{S\subseteq [n],|S|\geq 2}\,\sum_{i\in S}q_i\left(\frac{y_{i,0}}{y_{i,1}}+\frac{y_{i,1}}{y_{i,0}}\right)\right).
$$

\noindent \textbf{A consequence of Theorem~\ref{thm:bunkbed} about bicolored Cayley trees.}\\
The consequences of Theorem~\ref{thm:bunkbed} explored in this paper are mainly about Cartesian products of graphs. 
Let us mention, for fun, a consequence with a different flavor. Let $K_{p,p}$ be the complete bipartite graph with black vertices labeled $1,\ldots,p$ and white vertices labeled $1',\ldots,p'$. Let $m<p$ and let $R_m$ be the set of rooted spanning trees of $K_{p,p}$ containing the edge $\{i,i'\}$ for all $i\in[m]$. Observe that the complete bipartite graph $K_{p,p}$ is obtained from the strong product $K_p\boxtimes K_2$ by erasing all the straight arcs. Accordingly, we say that the \emph{spin} of the edge $\{i,i'\}$ is 0 if it is oriented toward the black endpoint $i$, and is 1 otherwise. Then, a consequence of Theorem~\ref{thm:bunkbed} is that for a uniformly random rooted tree in $R_m$ the spins of the edges $\{i,i'\}$ are independent and uniformly distributed. We do not know of an elementary proof of this fact.\\

\noindent \textbf{Conjecture for the spins of forests for Cartesian products of complete graphs.}\\
Just as formula~\eqref{eq:spanning-cube-2} was suggestive of the independence property for the spins of a random forest of the hypercube, formula~\eqref{eq:product-complete} and Proposition~\ref{prop:Laplacian} suggest an independence property that we make explicit now. 
Let $G=(U,A)$ be a weighted digraph, let $K_p$ be the complete graph with vertex set $[p]$, and let $H=G\times K_p$ be their Cartesian product. For $u\in U$ and $(i,j)\in[p]$ we call the arc of $H$ going from $(u,i)$ to $(u,j)$ a \emph{vertical arc of spin $j$ at vertex $u$}. We say that two rooted forests of $H$ have the same $G$-projection if they have the same number of $a$-arcs for all $a\in A$ and they have the same number of vertical arcs at $u$ for all $u\in U$. The \emph{multispin} of a rooted forest $F$ of $H$ at a vertex $u\in U$ is the multiset of the spins of the vertical arcs of $F$ at $U$. We now conjecture an analogue of Theorem~\ref{thm:bunkbed}:
\begin{conj}\label{conj:multispin-indep}
Let $F_0$ be a rooted forest of $H=G\times K_p$, and let $F$ be a uniformly random rooted forest of $H$ conditioned to have the same $G$-projection as $F_0$. Then the multispins $(\si_u)_{u\in U}$ of the random forest $F$ at the different vertices of $G$ are independent. 
\end{conj}
Observe that Conjecture~\ref{conj:multispin-indep} would readily implies Corollary~\ref{cor:product-complete} (in the same way as Theorem~\ref{thm:bunkbed} implied~\eqref{eq:spanning-cube-2}). It should also be mentioned that a stronger conjecture is false: it is not true that the subforests $(F_u)_{u\in U}$ made of the vertical arcs of the random forest $F$ at the different vertices of $G$ are independent (indeed, one can find a counterexample for $H=K_3\times K_3$).\\



\noindent \textbf{Acknowledgements.} I owe many thanks to Prasad Tetali for presenting to me the problem of finding a combinatorial proof of~\eqref{eq:spanning-cube}, and for pointing out that~\eqref{eq:no1} had a neat combinatorial interpretation. I also thank Victor Reiner and Richard Stanley for stimulating discussions, and Glenn Hurlbert and Bojan Mohar for useful references.

\bibliographystyle{plain} 
\bibliography{biblio-spanning-cube}

\begin{thebibliography}{1}

\bibitem{AZ}
M.~Aigner and G.M. Ziegler.
\newblock {\em Proofs from THE BOOK}.
\newblock Springer-Verlag, fourth edition, 2011.

\bibitem{Fiedler:eigenvalues-Cartesian}
M.~Fiedler.
\newblock Algebraic connectivity of graphs.
\newblock {\em Czech. Math. J.}, 23(98):298--305, 1973.

\bibitem{Hurlbert:spanning-trees}
G.H. Hurlbert.
\newblock On encodings of spanning trees.
\newblock {\em Discrete Appl. Math.}, 155(18):2594–2600, 2007.

\bibitem{Martin-Reiner:spanning-cube}
J.L. Martin and V.~Reiner.
\newblock Factorizations of some weighted spanning tree enumerators.
\newblock {\em J. Combin. Theory, Ser. A}, 104(2):287--300., 2003.

\bibitem{Stanley:vol2}
R.P. Stanley.
\newblock {\em Enumerative combinatorics, volume 2}.
\newblock Cambridge University Press, 1999.

\bibitem{Tuffley:spanning-3-cube}
C.~Tuffley.
\newblock Counting the spanning trees of the 3-cube using edge slides, 2011.
\newblock ArXiv:1109.6393.

\bibitem{Zeilberger:combinatorial-matrix}
D.~Zeilberger.
\newblock A combinatorial approach to matrix algebra.
\newblock {\em Discrete Math.}, 56:61--72, 1985.

\end{thebibliography}


\end{document}